\documentclass{article}
\usepackage{graphicx,amsmath,amsthm,amssymb} % Required for inserting images
\usepackage{xcolor}
\usepackage{fullpage}
\usepackage{enumitem}
\usepackage{tikz}

\title{On the Domatic Game}
\author{Sean English\footnote{University of North Carolina Wilmington. \texttt{EnglishS@uncw.edu}} \and London Swan\footnote{University of North Carolina Wilmington. \texttt{Lcs2908@uncw.edu}}}
\date{}
\usepackage{hyperref}
\newtheorem{theorem}{Theorem}[section]
\newtheorem{corollary}[theorem]{Corollary}
\newtheorem{proposition}[theorem]{Proposition}
\newtheorem{lemma}[theorem]{Lemma}
\newtheorem{claim}[theorem]{Claim}

\newtheorem{question}[theorem]{Question}

\newtheorem{conjecture}[theorem]{Conjecture}
\newtheorem{observation}[theorem]{Observation}
\newtheorem{definition}{Definition}

\DeclareMathOperator{\domg}{dom_g}
\DeclareMathOperator{\dom}{dom}
\DeclareMathOperator{\score}{score}

\begin{document}
	
	\maketitle
	\begin{abstract}
		The domatic game with pallete size $k$ is a $2$-player game played on a graph $G$ recently introduced by Hartnell and Rall. Players Alice and Bob take turns choosing an uncolored vertex from $G$, and coloring it a color from $\{1,2,\dots,k\}$. The game ends once all vertices in $G$ have been assigned a color. Alice wins if all $k$ colors induce a dominating set of $G$, and otherwise Bob wins. The domatic game number, $\domg(G,X)$ is the the largest pallete size $k$ such that Alice wins the domatic game when player $X$ goes first (where $X$ is either Alice or Bob).
		
		We prove for any graph $G$ of order $n$,
		\[
		\domg(G,X)=\Omega\left(\frac{\delta(G)}{\log n}\right).
		\]
		In addition, we show that for any $k$ there exists a graph $G$ with minimum degree $\delta(G)=k$ and $\domg(G,X)=1$, and there exists a graph $G'$ with $\domg(G',X)=1$ while having (non-game) domatic number $\dom(G')=k$. We explore how the domatic game number changes when changing who goes first, and when considering subgraphs of $G$. We also introduce a score variant of the domatic game, and use this to get bounds on the original domatic game.	
	\end{abstract}
	
	\section{Introduction}
	Let $G$ be a graph. A set $S\subseteq V(G)$ is a \textbf{dominating set} if every vertex in $V(G)$ is either in $S$ or adjacent to $S$. The \textbf{domination number} of $G$, $\gamma(G)$ is the size of the smallest dominating set in $G$. A \textbf{domatic partition} of $G$ is a partition $V(G)=S_1\cup S_2\cup\dots\cup S_k$ such that each $S_i$ is a dominating set in $G$. The \textbf{domatic number} of the graph $G$, denoted $\dom(G)$, is the size of the largest domatic partition. For a (somewhat dated) survey on domatic numbers, see~\cite{Z1998}, or for a more recent but less comprehensive look at the domatic number, see Chapter 12 from~\cite{HHH2023}.
	
	Motivated by the domatic number, the \textbf{domatic game}, recently introduced by Hartnell and Rall~\cite{HR2025}, is a $2$-player combinatorial game. In the domatic game on the graph $G$ with palette size $k$, players Alice and Bob take turns selecting vertices of $G$ and coloring them a color from $[k]$. The game ends once every vertex in $G$ has been assigned a color. Alice wins if for every $c\in [k]$, the set of vertices colored $c$ forms a dominating set of $G$, and Bob wins otherwise. There are two variations of this game, based on which player goes first. The \textbf{domatic game number} of a graph $G$ is the largest palette size that Alice has a strategy to win regardless of Bob's strategy. We write $\domg(G,A)$ for the domatic game number when Alice goes first, and when Bob goes first we denote the domatic game number by $\domg(G,B)$. We may also write $\domg(G,X)$, where $X\in \{A,B\}$, to talk about the game where we do not specify who goes first. First noted by Hartnell and Rall~\cite{HR2025}, it is straightforward to see that
	\[
	\domg(G,X)\leq \dom(G)
	\]
	for every graph $G$.
	
	In addition to studying the domatic game as introduced by Hartnell and Rall~\cite{HR2025}, we also introduce a score variant of the game which is useful for proving bounds on $\domg(G,X)$, and we believe of independent interest as well. The \textbf{domatic score game} with palette size $k$ on the graph $G$ is played identically to the domatic game, however instead of players winning or losing, we assign a score to the end of the game. In particular, the \textbf{score} is the number of colors which induce a dominating set in $G$. Bob's goal is to minimize the score, whereas Alice's goal is to maximize the score. We denote by $\score(G,X,k)$ the score at the end of the domatic score game played on $G$ with palette size $k$ and Player $X$ moving first, when both players play optimally.
	
	\subsection{Definitions, Notation and Organization}
	
	For a positive integer $k$, we will write $[k]:=\{1,2,\dots,k\}$. Given a set $S$ and a non-negative integer $k$, we let $\binom{S}{k}$ denote the collection of all $k$-element subsets of $S$. For asymptotic expressions, we will use Bachmann-Landau notation, and all asymptotics will be as $n\to \infty$. We will write $\log$ without a base for the natural logarithm.
	
	The paper is organized as follows. In Section~\ref{section degree}, we will explore bounds between the domatic game number and minimum degree of a graph, and provide some other related bounds. In Section~\ref{section relationship between games}, we will explore bounds that relate one domatic game to another, for example when changing who goes first, or changing the palette size of the game. In Section~\ref{section subgraphs}, we will explore how the domatic game number interacts with subgraphs. Finally in Section~\ref{section concluding remarks}, we leave the reader with some open questions.

	\section{Degree and the Domatic Game}\label{section degree}

	Our first result deals with the relationship between the domatic game number and the minimum degree of the graph. It is straightforward to see that $\domg(G,X)\leq \dom(G)\leq \delta(G)+1$. Refining this, Hartnell and Rall proved the following.
	
	\begin{proposition}[\cite{HR2025}]\label{proposition upper bound in terms of min degree}
		For any graph $G$ on $n$ vertices,
		\[
		\domg(G,X)\leq \frac{\delta(G)+3}{2}.
		\]
		Furthermore, there are examples of graphs where this bound is tight.
	\end{proposition}

	It is worth noting that Hartnell and Rall's result is actually slightly stronger than this - you can do slightly better with more information about the parity of $n$ and knowing who goes first. In particular, if $G$ is regular of odd degree and Alice goes first, the bound above can be improved by $1$.
	
	Our result will focus on a lower bound in terms of the minimum degree. Zelinka~\cite{Z1983} exhibited an family of graphs with unbounded minimum degree where every graph in the family had domatic number 2. Thus, increasing minimum degree alone is not sufficient to guarantee the domatic number or domati game number increases. The examples given by Zelinka had ``large'' maximum degree relative to minimum degree. Feige, Halld\'orsson, Kortsarz and Srinivasan~\cite{FHKS2002} prove the following, showing that the gap between minimum degree and maximum degree is necessary to have small domatic number but large minimum degree. 
	
	\begin{theorem}[\cite{FHKS2002}]\label{theorem domatic number minimum degree maximum degree lower bound}
		For any graph $G$,
		\[
		\dom(G)=\Omega\left(\frac{\delta(G)}{\log \Delta(G)}\right),
		\]
		and there exists graphs where the above bound is tight. 
	\end{theorem}
	
	Our first result extends this to the game setting with a slight weakening of the bound.
	
	\begin{theorem}\label{theorem main theorem on degree}
		For any graph $G$ on $n$ vertices,
		\[
		\domg(G,X)=\Omega\left(\frac{\delta(G)}{\log n}\right).
		\]
		Furthermore, for all $n$ there exists a graph $G_n$ on $n$ vertices with $\delta(G_n)=\Theta(\log n)$ and $\domg(G_n,X)=1$.
	\end{theorem}
	
	Theorem~\ref{theorem main theorem on degree} will follow immediately from Theorem~\ref{theorem graphs with large minimum degree but small domatic game number} and Theorem~\ref{theorem minimum degree formal theorem}. We note that in Section~\ref{section concluding remarks}, we make some comments on the differences in the bounds in Theorems~\ref{theorem domatic number minimum degree maximum degree lower bound} and~\ref{theorem main theorem on degree}.
	
	We need the following result, which gives an easy structural gadget we can use to force the domatic game nubmer to be $1$.
	
	\begin{proposition}[\cite{HR2025}]\label{proposition adjacent to two degree 1 implies game number 1}
		If $G$ is a graph which contains a vertex adjacent to two degree $1$ vertices, then $\domg(G,X)=1$.
	\end{proposition}
	
	We start with the following.
	
	\begin{theorem}\label{theorem graphs with large minimum degree but small domatic game number}
		Let $n\in\mathbb{N}$. There exists a graph $G_n$ on $n$ vertices with $\delta(G_n)=\Omega(\log n)$ such that $\domg(G_n,X)=1$.
	\end{theorem}
	
	We note that the construction below is based on the construction in~\cite{Z1983} of a graph with large minimum degree and (non-game) domatic number $2$.
	
	\begin{proof}
		First we will describe a construction $G_n$ on $n$ vertices with $\delta(G_n)=\Omega(\log n)$.  For each $k\in\mathbb{N}$, $k\geq 2$, define $n_k:=2(k+2)+\binom{2(k+2)}{k}$. 
		
		Let $G_1:=K_1$ and $G_2:=2K_1$. If $3\leq n<n_2=36$, then let $G_n$ be any graph on $n$ vertices with $\delta(G_n)\geq 1$ which contains a vertex adjacent to two degree $1$ vertices (for example $K_{n-2}$, with two additional vertices joined to the same vertex in the clique). 
		
		Now, if $n\geq n_2$, let $k$ be the largest integer such that $n\geq n_k$. Let $S:=[2(k+2)]$, $\mathcal{T}:=\binom{[2(k+2)]}{k}$ and $U$ a set of size $n-n_k$ disjoint from $S$ and $\mathcal{T}$. Let $G_n$ be the bipartite graph with partite sets $S$ and $\mathcal{T}\cup U$ with the following edges: If $s\in S$ and $T\in \mathcal{T}$, $sT\in E(G_{n_k})$ if and only if $s\in T$. If $s\in S$ and $u\in U$, then $su\in E(G_n)$ if and only if $1\leq s\leq k$. See Figure~\ref{figure drawing of G134} for a sketch of $G_n$ when $n=134$.
		
		We note that
		\[
		|V(G_n)|=|S|+|\mathcal{T}|+|U|=2(k+2)+\binom{2(k+2)}{k}+n-n_k=n_k+n-n_k=n.
		\]
		Furthermore, $\delta(G_n)=k$ as every vertex in $\mathcal{T}$ or $U$ has degree $k$, whereas the vertices in $S$ have degree at least $\binom{2(k+2)-1}{k-1}\geq k$. Furthermore, by the definition of $k$, we have that
		\[
		n<n_{k+1}=2(k+3)+\binom{2(k+3)}{k}\leq \binom{2(k+3)}{k+3}\leq 4^{k+3},
		\]
		so in particular
		\[
		\delta(G_n)=k\geq \log_4 n-3=\Omega(\log n).
		\]
		
		\begin{figure}
			\begin{center}
				\begin{tikzpicture}
					\draw[fill=black] (-0.5,-2) circle (2pt);
					\draw[fill=black] (0.5,-2) circle (2pt);
					\draw[fill=black] (1.5,-2) circle (2pt);
					\draw[fill=black] (2.5,-2) circle (2pt);

					\draw[fill=black] (0,0) circle (2pt);
					\draw[fill=black] (1,0) circle (2pt);
					\draw[fill=black] (2,0) circle (2pt);
					\draw[fill=black] (3,0) circle (2pt);
					\draw[fill=black] (4,0) circle (2pt);
					\draw[fill=black] (5,0) circle (2pt);
					\draw[fill=black] (6,0) circle (2pt);
					\draw[fill=black] (7,0) circle (2pt);
					\draw[fill=black] (8,0) circle (2pt);
					\draw[fill=black] (9,0) circle (2pt);

					\draw[fill=black] (0.5,3) circle (2pt);
					\draw[fill=black] (1.5,3) circle (2pt);
					\draw[fill=black] (2.5,3) circle (2pt);
					\draw[fill=black] (3.5,3) circle (2pt);
					\draw[fill=black] (4.5,3) circle (2pt);
					\draw[fill=black] (5.5,3) circle (2pt);
					\draw[fill=black] (-0.5,3) circle (2pt);
					\draw[fill=black] (-1.5,3) circle (2pt);
					\draw[fill=black] (8.5,3) circle (2pt);
					\draw[fill=black] (9.5,3) circle (2pt);
					\draw[fill=black] (10.5,3) circle (2pt);
					
					\node at (-1.5,3.5) {\small$\{1,2,3\}$};
					
					\node at (2.5,3.5) {\small$\{1,2,7\}$};
					
					\node at (0.5,3.5) {$\textbf{\dots}$};
					
					\node at (4.5,3.5) {$\textbf{\dots}$};
					
					\node at (10.5,3.5) {\small$\{8,9,10\}$};
					
					\node at (-.3,0) {$1$};
					\node at (.7,0) {$2$};
					\node at (1.7,0) {$3$};
					\node at (2.7,0) {$4$};
					\node at (3.7,0) {$5$};
					\node at (4.7,0) {$6$};
					\node at (5.7,0) {$7$};
					\node at (6.7,0) {$8$};
					\node at (7.7,0) {$9$};
					\node at (8.7,0) {$10$};
					
					\node at (-2,3) {$\mathcal{T}$ {\large$[$}};
					\node at (11,3) {{\large$]$}};
					
					\node at (-0.8,0) {$S$ {\large$[$}};
					\node at (9.5,0) {{\large$]$}};

					\node at (-1,-2) {$U$ {\large$[$}};
					\node at (3,-2) {{\large$]$}};

					\node at (7,3) {\large$\textbf{\dots}$};
					
					\draw (0,0)--(-0.5,-2)--(1,0)--(0.5,-2)--(2,0)--(1.5,-2)--(0,0);
					
					\draw (0,0)--(0.5,-2);
					\draw (1,0)--(1.5,-2);
					\draw (2,0)--(-0.5,-2);
					\draw (0,0)--(2.5,-2)--(1,0);
					\draw (2,0)--(2.5,-2);

					\draw (0,0)--(-1.5,3)--(1,0);
					\draw (-1.5,3)--(2,0);
					
					\draw (0,0)--(2.5,3)--(1,0);
					\draw (2.5,3)--(6,0);
					\draw (7,0)--(10.5,3)--(8,0);
					\draw (9,0)--(10.5,3);
				\end{tikzpicture}
				\caption{A drawing of $G_{134}$. Since $n_3\leq 134<n_4$, we have that $k=3$. This gives us that $|\mathcal{T}|=\binom{10}{3}=120$, while $|S|=10$ and $|U|=134-n_3=4$.}\label{figure drawing of G134}
			\end{center}
		\end{figure}
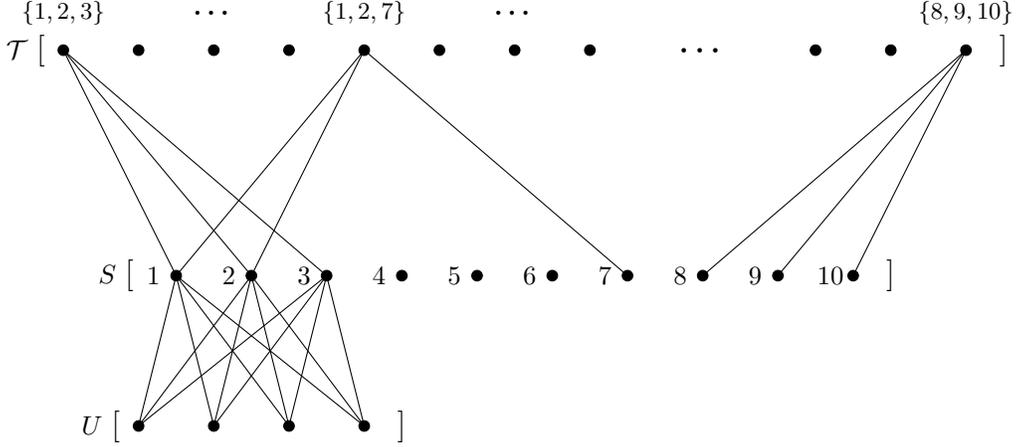

		Now we bound the domatic game number of $G_n$. If $n=1$ or $n=2$, the result is immediate. If $3\leq n<n_2$, then $G_n$ contains a vertex adjacent to two degree $1$ vertices, so by Proposition~\ref{proposition adjacent to two degree 1 implies game number 1} we are done. If $n_k\leq n<n_{k+1}$ for some $k\geq 2$, we use the following strategy for Bob in the $(G_n,X,2)$-game:
		
		For the first $k+2$ turns, Bob will arbitrarily choose an uncolored vertex from $S$ and color it $1$. On turn $k+3$, Bob will choose an uncolored vertex $T^*\in \mathcal{T}$ such that every vertex in $N(T^*)$ is colored $1$, and color $T^*$ the color $1$. This is always possible since at this stage there are at least $k+2$ vertices in $S$ colored $1$, and thus there are at least $\binom{k+2}k>k+2$ vertices $T\in \mathcal{T}$ whose neighborhood is completely colored $1$, at least one of which has not been colored yet. Thus, Alice loses the $(G_n,X,2)$-game, so $\domg(G_n,X)=\domg(G_n,X)=1$.	
	\end{proof}
	
	In the proof above, we had that $G_{n_k}$ was a graph with $\delta(G_{n_k})=k$ and $\domg(G,X)=1$. This gives us the following corollary, showing that any integer can be realized as the minimum degree of a graph with domatic game number $1$ (regardless of whom goes first).
	
	\begin{corollary}
		For all $k\in\mathbb{N}$, there exists a graph $G$ such that $\delta(G)=k$ and $\domg(G,X)=1$.
	\end{corollary}
	
	In the proof of Theorem~\ref{theorem main theorem on degree}, we'll need the following classical result on Maker-Breaker games. Recall that the Maker-Breaker game on the hypergraph $H$ is a $2$-player game played between players Maker and Breaker, where players take turns selecting previously unselected vertices from $H$. If at some point Maker selects all the vertices in some edge of $H$, Maker wins. Otherwise, Breaker wins. We make use of the following classical result, proven by Erd\H{o}s and Selfridge~\cite{ES1973}
	
	\begin{theorem}[\cite{ES1973}]\label{theorem Erdos selfridge}
		Let $H$ be a hypergraph. If
		\[
		\sum_{e\in E(H)}2^{-|e|}<\frac{1}{2}
		\]
		then Breaker has a winning strategy in the unbiased Maker-Breaker game on $H$ when either Maker or Breaker goes first.
	\end{theorem}
	
	The following is a direct consequence of Theorem~\ref{theorem Erdos selfridge}, but in a form which will be more useful for our purposes.
	
	\begin{corollary}\label{corollary our use of maker breaker}
		Let $H$ be a hypergraph with at most $n$ edges, each of size at least $\log_2 n+2$. Then Breaker has a winning strategy on $H$ in the unbiased Maker-Breaker game, even if Maker can skip turns.
	\end{corollary}
	
	\begin{proof}
		First we note that it is never optimal for Maker to skip a turn, so giving them the power to do so does not change which games are Maker-win vs. Breaker-win. Now, we can verify
		\[
		\sum_{e\in E(H)}2^{-|e|}\leq n\cdot 2^{-\log_2 n-2}=\frac{1}{4}<\frac{1}{2}.
		\]
		Thus, by Theorem~\ref{theorem Erdos selfridge}, Breaker has a winning strategy in the unbiased Maker-Breaker game regardless of who goes first.
	\end{proof}
	
	We need the following technical lemma, which will help provide a crucial estimate in the final proof of Theorem~\ref{theorem main theorem on degree}.
	
	\begin{lemma}\label{lemma increasing function estimate}
		Let $n\geq 2$ be an integer and $1\leq r\leq n$, $0<c\leq \frac{\log 2}{4}$ real numbers. Set $k:=\lceil cr\rceil $ and define $f:\mathbb{N}\cup \{0\}\to\mathbb{R}$ by
		\[
		f(i)=\left(\frac{e\cdot r\log n}{i(k-1)}\right)^i.
		\]
		Then $f$ is increasing for $i\leq \log_2 n+2$.
	\end{lemma}
	
	\begin{proof}
		We have that
		\[
		\frac{f(i)}{f(i+1)}=\frac{\left(\frac{e\cdot r\log n}{i(k-1)}\right)^i}{\left(\frac{e\cdot r\log n}{(i+1)(k-1)}\right)^{i+1}}=\left(1+\frac{1}{i}\right)^i\frac{(i+1)(k-1)}{e\cdot r\log n}
		\]
		Note that $(1+1/i)^i\leq e$ for all $i\in\mathbb{N}$, and $k-1\leq cr$, so we have 
		\[
		\left(1+\frac{1}{i}\right)^i\frac{(i+1)(k-1)}{e\cdot r\log n}\leq \frac{c(i+1)}{\log n}.
		\]
		Then since $c\leq \frac{\log 2}{4}\leq\frac{\log n}{\log_2 n+3}$, the above expression is at most $1$ when $i\leq \log_2 n+2$, so $f(i)$ is increasing in the desired range.
	\end{proof}
	
	We are now ready to finish the proof of Theorem~\ref{theorem main theorem on degree}. The result below (along with Theorem~\ref{theorem graphs with large minimum degree but small domatic game number}) immediately implies Theorem~\ref{theorem main theorem on degree}.
	
	\begin{theorem}\label{theorem minimum degree formal theorem}
		There exists a constant $c>0$ such that if $G$ is a graph on $n\geq 2$ vertices, then
		\[
		\domg(G,X)\geq c\cdot \frac{\delta(G)}{\log n}.
		\]
	\end{theorem}
	
	\begin{proof}
		We will prove this only in the case when $n$ is large enough (i.e. when there exists some fixed $n_0$ such that the proof works for all $n\geq n_0$), noting that this suffices, since we may always choose $c$ small enough that for $n<n_0$, $c\cdot \frac{\delta(G)}{\log n}<1$, making the result true. 
		
		Let $c>0$ be a constant which we will specify later. Let $r:=\frac{\delta(G)}{\log n}$ and $k:=\lceil cr\rceil$. If $k<2$, the result is trivial, so assume $k\geq 2$. Let $H^*$ denote the hypergraph with $V(H^*)=V(G)$ and
		\[
		E(H^*)=\{N[v]\mid v\in V(G)\}.
		\]
		Note that $H^*$ has at most $n$ edges, and each edge of $H^*$ is size at least $r\log n$. Let $H$ denote the hypergraph formed by replacing each edge $e\in E(H^*)$ with an edge $f\subseteq e$ such that $|f|=r\log n$ (if there are two edges $e,e'\in E(H^*)$ such that this process results in the same $f$ for both, then we can just add this single $f$ to $E(H)$). Then $H$ is a $(r\log n)$-uniform hypergraph with at most $n$ edges.
		
		\begin{claim}
			For $n$ large enough, there is a partition of $V(H)$ into $k$ parts, $V(H)=V_1\cup V_2\cup \dots\cup V_k$ such that $|e\cap V_i|\geq \log_2 n+2$ for all $e\in E(H)$, $i\in [k]$.
		\end{claim} 
		
		\begin{proof}
			Consider the random partition $V(H)=V_1\cup \dots\cup V_k$ such that each $v$ is uniformly and independently sorted into one of the sets $V_i$. Let $e\in E(H)$ and $i\in[k]$ be a fixed edge and index. We then calculate
			\begin{align*}
				\mathrm{Pr}(|e\cap V_i|<\log_2 n+2)&\leq \frac{1}{k^n}\sum_{i=0}^{\lfloor \log_2 n+2\rfloor}\binom{r\log n}{i}(k-1)^{r\log n -i}k^{n-r\log n}\\
				&=\left(1-\frac{1}{k}\right)^{r\log n}\sum_{i=0}^{\lfloor \log_2 n+2\rfloor}\frac{\binom{r\log n}{i}}{(k-1)^i}.
			\end{align*}
			Now, using the estimates $1-x\leq e^{-x}$, $\binom{a}{b}\leq \left(\frac{a\cdot e}{b}\right)^b$, along with Lemma~\ref{lemma increasing function estimate} (which applies as long as $c\leq \frac{\log 2}{4}$), we may write the following
			\begin{align*}
				\left(1-\frac{1}{k}\right)^{r\log n}\sum_{i=0}^{\lfloor \log_2 n+2\rfloor}\frac{\binom{r\log n}{i}}{(k-1)^i}&\leq \left(e^{-1/k}\right)^{r\log n}\sum_{i=0}^{\lfloor \log_2 n+2\rfloor}\left(\frac{e\cdot r\log n}{i(k-1)}\right)^i\\
				&\leq n^{-r/k}(\log_2 n+3)\left(\frac{e\cdot r\log n}{(\log_2 n+2)(k-1)}\right)^{\log_2 n+2}\\
				&\leq n^{-r/k}(\log_2 n+3)\left(\frac{e\log 2\cdot r}{k-1}\right)^{\log_2 n+2}.
			\end{align*}
			Let $\mathcal{E}$ be the event that there exists some $e\in E(H)$ and $i\in [k]$ such that $|e\cap V_i|< \log_2 n+2$. Using the union bound, we have that
			\[
			\mathrm{Pr}(\mathcal{E})\leq nk\cdot n^{-r/k}(\log_2 n+3)\left(\frac{e\log 2\cdot r}{k-1}\right)^{\log_2 n+2}.
			\]
			Now, using the estimates that $k\leq n$ and $\log_2 n+3\leq n$ for $n$ large, and using that for $k\geq 2$, $\frac{r}{k-1}\leq \frac{2}{c}$, we may write
			\begin{align*}
				nk\cdot n^{-r/k}(\log_2 n+3)\left(\frac{e\log 2\cdot r}{k-1}\right)^{\log_2 n+2}&\leq n^{3-1/c} \left(\frac{2e\log 2}{c}\right)^{\log_2 n+2}\\
				&=(2e\log 2)^2c^{-2}n^{3-1/c} n^{\log_2(2e\log 2)-\log_2(c)}\\
				&=(2e\log 2)^2c^{-2}n^{3+2e\log 2-1/c-\log_2(c)}.
			\end{align*}
			As $c\to 0^+$, we have that $-1/c-\log_2(c)\to -\infty$, so in particular we may choose $c>0$ small enough such that $3+2e\log 2-1/c-\log_2(c)<0$ (for example $c\leq 1/100$). Thus, if we choose $c=\min \{1/100,\log 2/4\}=1/100$, we have that $\mathrm{Pr}(\mathcal{E})\to 0$ as $n\to\infty$, so for $c=1/100$ and $n$ large enough, the desired partition exists.
		\end{proof}
		
		Now let us describe a winning strategy for Alice. For $i\in [k]$, let $H_i$ be the hypergraph induced by $H$ on $V_i$ (i.e. $V(H_i)=V_i$ and $E(H_i)=\{e\cap V_i\mid e\in E(H)\}$). Note that $H_i$ is a hypergraph with at most $n$ edges, each of size at least $\log_2 n+2$, so by Corollary~\ref{corollary our use of maker breaker}, Breaker has a winning strategy on $H_i$, even if Maker can skip turns. Alice will play according to these strategies as follows:
		\begin{itemize}
			\item If on Bob's last move, Bob plays in $V_i$ and there is an uncolored vertex in $V_i$, then according to Breaker's strategy on $H_i$, Alice will color a vertex in $V_i$ the color $i$.
			\item Otherwise, Alice will arbitrarily choose some $i$ such that $V_i$ has an uncolored vertex, and play according to Breaker's strategy on $H_i$ (interpreting this as a situation where Maker skipped their turn).
		\end{itemize}
		Alice as Breaker wins all $k$ of these Maker-breaker games, so every edge in $H$ receives every color in $[k]$. Thus, Alice wins the $(G,X,k)$-game as long as $n$ is large enough.
	\end{proof}
	
	One might wonder if the domatic game number and the domatic number can differ by much - theoretically if these two numbers were always very close to each other, then Theorem~\ref{theorem domatic number minimum degree maximum degree lower bound} would immediately imply Theorem~\ref{theorem main theorem on degree}. To motivate the significance of Theorem~\ref{theorem main theorem on degree} (and for independent interest), we present the following theorem, which shows the domatic number can grow even among graphs whose domatic game number is $1$.
	
	\begin{theorem}\label{theorem domatic number and domatic game number are far apart}
		For any $k\in\mathbb{N}$, there exists a graph $G$ with $\dom(G)\geq k$ and $\domg(G,X)=1$.
	\end{theorem} 
	
	\begin{proof}
		Fix $k\in\mathbb{N}$ and let $T$ be a binary rooted tree with root $r$ and height $k-1$. For a vertex $x\in V(T)$, let $\lceil x\rceil$ denote the set of all vertices $y\in V(T)$ such that $y$ is in the $x\text{-}r$-path in $T$, and let $\lfloor x\rfloor$ denote all the vertices $y\in V(T)$ such that $x\in \lceil y\rceil$. Recall that we say $x$ is a \textbf{child} of $y$ in $T$ if $xy\in E(T)$ and $x\in \lceil y\rceil$.
		
		Let $G'$ be the graph with $V(G')=V(T)$ and such that $xy$ is an edge of $G'$ if and only if either $x\in \lfloor y\rfloor$ or $y\in \lfloor x\rfloor$. Consider a $(G',B,2)$-game on $G'$. Bob will use the following strategy: 
		\begin{itemize}
			\item Bob will only use color 1.
			\item For $1\leq i\leq k$, on Bob's $i$th turn, Bob will color a vertex $v$ at distance (measured in $T$) $i-1$ from $r$ with $\lfloor v \rfloor$ completely uncolored and such that every vertex $v'\in \lceil v \rceil\setminus \{v\}$ is color 1.
			\item After the $k$th turn, Bob plays arbitrarily.
		\end{itemize}
		Note that Bob can always find a vertex $v$ as described above: on the first turn, Bob can play at the root $r$, and then proceeding by induction, if Bob played at $v_{i-1}$ on turn $i-1$, then on the $i$th turn, at least one of the two children (as measured in $T$) of $v_{i-1}$ satisfies the requirements. 
		
		On Bob's $k$th turn, he colors a vertex $v$ that is at distance $k-1$ (in $T$) from $r$ the color $1$, resulting in every vertex $\lceil v\rceil$ being color $1$. However, the vertices in $T$ at distance $k-1$ from $r$ are leaves, so by the definition of $G$, $\lceil v\rceil =N_{G'}[v]$, so at the end of the game $v$ does not see color $1$, and thus Bob wins. This gives us that $\domg(G',B)=1$.
		
		Now, we will prove the theorem with $G:=2G'$. In the $(G,X,2)$-game, regardless of if $X=A$ or $X=B$, Bob can implement a winning $(G',B,2)$-strategy on one of the two components of $G$ (if Alice goes first, Bob implements the strategy on the component that Alice did not play on, if Bob goes first, Bob can implement it on either component). Thus Bob wins the $(G,X,2)$-game, giving us $\domg(G,X)=1$.
		
		To see that $\dom(G)\geq k$, we note that in $G'$, for any fixed $0\leq \ell\leq k-1$, the vertice at distance (measured in $T$) exactly $\ell$ from the root $r$ constitute a dominating set, giving us $k$ pairwise-disjoint dominating sets.
	\end{proof}

	We conclude this section with a discussion on general upper bounds on the domatic game number. Recall that in~\cite{HR2025}, Hartnell and Rall proved Proposition~\ref{proposition upper bound in terms of min degree}, an upper bound on the domatic game number in terms of minimum degree. We give here an upper bound in terms of the domination number, rather than the degree of a graph.
	
	\begin{proposition}\label{proposition upper bound in terms of gamma}
		For any graph $G$ on $n$ vertices,
		\begin{equation}\label{equation upper bound in terms of domination}
			\domg(G,X)\leq\frac{n}{2\gamma(G)}+1
		\end{equation}
	\end{proposition}
	
	\begin{proof}
		Let $k:=\left\lfloor \frac{n}{2\gamma(G)}\right\rfloor+1$. Consider an $(G,X,k+1)$-game. If Alice goes first, assume without loss of generality that the first color Alice plays is the color $1$. Consider a strategy for Bob where Bob uses only the color $1$, and otherwise plays arbitrarily. This leads to at least $n/2$ vertices being color $1$ at the end of the game. The remaining $n/2$ vertices can be partitioned into at most $\frac{n}{2\gamma(G)}$ dominating sets. Thus, only $k$ colors end up as dominating sets, so Bob wins the $(G,X,k+1)$-game, leading to the bound.
	\end{proof}
	
	It is worth noting that the upper bound in Proposition~\ref{proposition upper bound in terms of gamma} can be thought of in terms of degrees - a natural bound on the domination number of a graph $G$ on $n$ vertices is $\gamma(G)\geq \frac{n}{\Delta(G)+1}$ (this bound was perhaps first noted in writing in~\cite{WAS1979}). Applying this to~\eqref{equation upper bound in terms of domination} yields
	\[
	\domg(G,X)\leq \frac{\Delta(G)+1}{2}+1=\frac{\Delta(G)+3}{2},
	\]
	recovering the bound in Proposition~\ref{proposition upper bound in terms of min degree} if $G$ is regular (but slightly worse than Hartnell and Rall's extension in the case where $G$ is regular of odd degree and Alice goes first), and giving us a stronger bound whenever $G$ is regular and the bound $\gamma(G)\geq \frac{n}{\Delta(G)+1}$ is not tight.
	
	We note that the proofs of Propositions~\ref{proposition upper bound in terms of min degree} and~\ref{proposition upper bound in terms of gamma} did not actually require the palette size of the game to be fixed, and thus these bounds apply to the score game as well.
	
	\begin{corollary}\label{corollary absolute upper bounds on score}
		For any graph $G$ on $n$ vertices, and any $\ell\in\mathbb{N}$,
		\[
		\score(G,X,\ell)\leq \min\left\{\frac{\delta(G)+3}{2},\frac{n}{2\gamma(G)}+1\right\}
		\]
	\end{corollary}
	
	\section{Relationships Between Different Domatic Games}\label{section relationship between games}
	
	In this section we explore bounds which feature two or more domatic games played on the same graph. The main exploration involves how adjusting the palette size of the game affects things.
	
	\subsection{Palette Sizes and the Score Game}
	
	In~\cite{HR2025}, motivating by an open question on the game chromatic number (see e.g.~\cite{H2025}) Hartnell and Rall asked the following question:
	
	\begin{question}[\cite{HR2025}]
		Does there exist a graph $H$ and a positive integer $k$ such that Bob wins the $(H,X,k)$-game, but Alice wins the $(H,X,k+1)$-game?
	\end{question}
	
	We answer this question in the negative.
	
	\begin{theorem}\label{theorem game monotonicity in pallete size}
		For any graph $G$ and $k,\ell\in\mathbb{N}$, if the $(G,X,k)$-game is Bob-win, then the $(G,X,k+\ell)$-game on $G$ is also Bob-win.
	\end{theorem}
	
	\begin{proof}
		Bob will use a ``colorblind'' strategy in the $(G,X,k+\ell)$-game: Formally we will play two games simultaneously, the main game, a $(G,X,k+\ell)$-game between Alice and Bob, and an auxiliary game, a $(G,X,k)$-game between players Alice2 and Bob2. Bob2 will play according to a winning strategy in the $(G,X,k)$-game, and Bob will copy every move Bob2 makes. Whenever Alice plays a color in $[k]$, Alice2 will copy Alice, and if Alice plays a color in $[k+\ell]\setminus [k]$, Alice2 will play the color $k$ at the same vertex that Alice just played at.
		
		Since Bob2 wins the $(G,X,k)$-game, at the end of the games, there is a color $i\in [k]$ such that $i$ does not induce a dominating set in the auxiliary game. If $i<k$, then this color also does not induce a dominating set in the main game, whereas if $i=k$, then actually none of the colors in $[k+\ell]\setminus [k]$ induce a dominating set in the main game. In either case, Bob wins the $(G,X,k+\ell)$-game.
	\end{proof}
	
	The crux of the above proof is essentially that expanding the palette size does not give Alice any extra power, it only is a detriment to her. However, when we consider the score variant of the domatic game, it is less clear if increasing the palette size benefits one player over another, at least when the palette size is greater than the domatic game number. When the palette size is small though, we have a strong understanding of the score.
	
	\begin{observation}\label{observation score and pallete size are the same when pallete size less than dom}
		Let $k\leq \domg(G,X)$. Then $\score(G,X,k)=k$.
	\end{observation}
	
	\begin{proof}
		Theorem~\ref{theorem game monotonicity in pallete size} implies that Alice has a winning strategy in the $(G,X,k)$-game, so Alice can use this strategy to score $k$ points in the $(G,X,k)$-score-game.
	\end{proof}
	
	When $k>\domg(G,X)$, the relationship between the palette size and the score is less clear. In particular, increasing the palette size theoretically could give a boost to either player - Bob has more freedom to choose colors which are less synergistic with previously chosen colors, and Alice has more colors with which to potentially create dominating sets. We can prove that the score cannot decrease by much with a palette size larger than $\domg(G,X)$.
	
	\begin{proposition}\label{proposition score does not decrease much}
		Let $\ell>\domg(G,X)$. Then $\score(G,X,\ell)\geq \domg(G,X)-1$.
	\end{proposition}
	
	\begin{proof}
		To prove this, we first introduce an intermediary parameter. We will call a Bob strategy in which Bob only ever plays the color $1$, a \textbf{$1$-Bob strategy}. For an integer $s$, let $r_s$ denote the score of the $(G,X,s)$-score-game where Bob is restricted to playing a $1$-Bob strategy, but otherwise Alice and Bob play optimally (and for the sake of well-definedness, we will assume Alice is aware of Bob's restriction when Alice chooses their strategy).
		
		We claim that $r_t\leq r_s$ for any $t\leq s$. Indeed, in the $(G,X,s)$-score game, Alice can play according to their optimal strategy from the $(G,X,t)$-score game (Bob never plays a move outside of the palette $[t]$ since Bob is playing a $1$-Bob strategy, so this strategy is well-defined), and score at least $r_t$ points in the $(G,X,s)$-score game. We further note that for any $t\in\mathbb{N}$ $\score(G,X,t)\leq r_t$ since Bob has a $1$-Bob strategy (and thus a strategy) which restricts the score to at most $r_t$ in the $(G,X,t)$-score game. 
		
		We now claim that $r_t\leq \score(G,X,t)+1$ for any $t\in \mathbb{N}$. Consider the following $1$-Bob strategy: We will play two $(G,X,t)$-score games simultaneously, the first (the main game) with Alice and Bob, where Alice is playing optimally and Bob plays according to the $1$-Bob strategy we are currently describing, and the second game with new players Alice2 and Bob2, where Bob2 is not restricted and plays optimally. On each of Alice's turns, Alice and Alice2 play exactly the same move (Alice2 copies Alice's move). On Bob's turns, Bob waits for Bob2 to play, and then copies the location Bob2 played, but plays the color $1$ regardless of which color Bob2 choose. After both games finish, their boards are only different in that all the places Bob2 played are colored $1$ in the main game. In the second game, the board ended with at most $\score(G,X,t)$ dominating colors. In the main game, each of these colors may still be dominating, and the color $1$ might additionally be dominating, but no other colors can be dominating, giving a $1$-Bob strategy that results in the final score of at most $\score(G,X,t)+1$, and thus $r_t\leq \score(G,X,t)+1$.
		
		Now, set $k:=\domg(G,X)$ and note that we have from Observation~\ref{observation score and pallete size are the same when pallete size less than dom} that $\domg(G,X)=\score(G,X,k)$. Putting everything together, we have
		\[
		\score(G,X,\ell)\geq r_\ell-1\geq r_k-1\geq\score(G,X,k)-1=\domg(G,X)-1.
		\]
	\end{proof}
	
	It is worth noting that the bound in Proposition~\ref{proposition score does not decrease much} is sharp infinitely often.
	
	\begin{proposition}\label{proposition example of score drop}
		For every $k\in\mathbb{N}$, there are infinitely many examples of graphs $G$ such that $\domg(G,X)=k$, while $\score(G,X,k+1)=k-1$.
	\end{proposition}
	
	\begin{proof}
		Fix $k\in\mathbb{N}$, and let $n\geq 3(2k-1)$ be any integer such that $(2k-1)\mid n$. Let $G:=\frac{n}{2k-1}K_{2k-1}$ denote the disjoint union of $n/(2k-1)$ copies of the clique $K_{2k-1}$. Let $K^{(1)},\dots,K^{(n/(2k-1))}$ denote the cliques in $G$. 
		
		First let us show that $\domg(G,X)=k$. When playing the $(G,X,k)$-game, Alice uses the following strategy:
		\begin{itemize}
			\item If Bob just played in a clique $K^{(i)}$, and there is an uncolored vertex in $K^{(i)}$, Alice plays in $K^{(i)}$. Otherwise, Alice chooses a clique with an uncolored vertex arbitrarily and plays there.
			\item Alice always chooses to play a color that is not currently featured in the clique she is playing on, unless the clique she is playing on already sees each color, in which case Alice plays a color arbitrarily.
		\end{itemize}
		With this strategy, for each clique $K^{(i)}$ one of two things happens: Either Alice plays on $K^{(i)}$ first, meaning Alice plays at least $\lceil (2k-1)/2\rceil=k$ times on $K^{(i)}$, each a different color, or Bob plays on $K^{(i)}$ first, meaning Alice plays on $K^{(i)}$ at least $\lfloor (2k-1)/2\rfloor=k-1$ times, playing a different color each time, and all those colors are different than the color Bob played first, giving this clique all $k$ colors. Thus, Alice wins the $(G,X,k)$-game.
		
		In the $(G,A,k+1)$-game, Bob uses the following strategy: Assume without loss of generality that Alice's first move plays the color $1$ in the clique $K^{(1)}$. Then in the next $k-1$ rounds, Bob plays the color $1$ in $K^{(1)}$, and after this Bob plays arbitrarily. $K^{(1)}$ ends up with at least $k$ vertices colored $1$, so $K^{(1)}$ sees at most $k$ colors in total, and so Bob wins.
		
		In the $(G,B,k+1)$-game, Bob plays arbitrarily in $K^{(1)}$, until at some point Alice plays a move in a clique that is not $K^{(1)}$. Note that since $|V(K^{(1)})|$ is odd, this is guaranteed to occur. Assume without loss of generality that Alice played the color $1$ in $K^{(2)}$. Then in the next $k-1$ rounds, Bob plays the color $1$ in $K^{(2)}$, and after this Bob plays arbitrarily. Similarly to in the $(G,A,k+1)$-game, this causes Bob to win.
		
		Thus, $\domg(G,X)=k$.
		
		Now, let us show that $\score(G,X,k+1)=k-1$. The lower bound follows from Proposition~\ref{proposition score does not decrease much}, so we need only give a Bob strategy that matches this. 
		
		First consider the $(G,A,k+1)$-score game, assume without loss of generality that Alice's first move is to play $1$ on the clique $K^{(1)}$. The strategy will have two ``phases''. In Phase 1, Bob plays $1$ on $K^{(1)}$ until all vertices in $K^{(1)}$ are colored. Note that until $K^{(1)}$ is completely colored, if Alice plays on any other clique, $K^{(1)}$ ends up with $\lceil 2k-1\rceil+1=k+1$ $1$'s, so $K^{(1)}$ received at most $k-1$ colors in total and Bob gets the desired score. Thus, we may assume that at the end of Phase 1, every move thus far has been played on $K^{(1)}$, and that $K^{(1)}$ sees exactly $k$ colors and in particular this implies that it is Bob's turn. Let us assume without loss of generality that  $K^{(1)}$ does NOT see the color $k+1$. Now, in Phase 2, Bob repeatedly plays the color $k+1$ on $K^{(2)}$ until every vertex on $K^{(2)}$ is colored. This results in at least $\lceil 2k-1\rceil$ vertices in $K^{(2)}$ being colored $k+1$, which implies that $K^{(1)}$ sees at most $k$ colors. Thus, $K^{(1)}$ and $K^{(2)}$ each see at most $k$ colors, but $K^{(1)}$ does not see $k+1$, so the cliques do not see the same $k$ colors. Thus, regardless of how play continues, the maximum score of the game is $k-1$, as desired.
		
		Now, in the $(G,B,k+1)$-game, we proceed similarly, but we need three ``phases''. In Phase 1, Bob plays $1$ continuously on $K^{(1)}$ until it is completely colored. Note that in this phase, if Alice plays two or more moves outside of $K^{(1)}$, then Bob ends up playing $\lceil 2k-1\rceil+1=k+1$ $1$'s on $K^{(1)}$, so $K^{(1)}$ receives at most $k-1$ colors, so we may assume Alice played at most one move outside $K^{(1)}$, in fact since $K^{(1)}$ has an odd number of vertices, Alice had to play a single move outside of $K^{(1)}$, say without loss of generality that Alice played color $r$ on $K^{(2)}$. In Phase 2, Bob plays the color $r$ on $K^{(2)}$. Similar to Phase 1 in the $(G,A,k+1)$-score game strategy, this forces Alice to only play on $K^{(2)}$, leaving us in a game state where $K^{(2)}$ receives exactly $k$ colors, and Bob is next to play. We may assume without loss of generality that the color $k+1$ does not appear on $K^{(2)}$, and similar as to Phase 2 in the $(G,A,k+1)$-score game, Phase 3 will consist of Bob repeatedly playing the color $k+1$ on $K^{(3)}$. This results in $K^{(3)}$ receiving exactly $k$ colors, but the colors on $K^{(2)}$ and $K^{(3)}$ do not match, so the score of the game is at most $k-1$.
		
		Thus, $\score(G,X,k+1)=k-1$. As such a graph $G$ on $n$ vertices exists for any $n\geq 3(k-1)$ with $(k-1)\mid n$, we indeed have an infinite family of examples.
	\end{proof}
	
	Propositions~\ref{proposition score does not decrease much} and~\ref{proposition example of score drop} show that the score can drop from increasing the palette size (but not by much). A natural question is if the score can increase, which we are not sure about (see Question~\ref{question score increase} in the concluding remarks). For a fixed graph $G$, the score cannot continuously increase as $\ell\to\infty$ since Corollary~\ref{corollary absolute upper bounds on score} gives absolute bounds on the score in terms of $G$. However, in the range where the score may increase, we can show that if the score does increase as $\ell$ increases, it does so a constant factor more slowly than $\ell$.
	
	\begin{proposition}\label{proposition upper bound on score constant factor more slowly than pallete size}
		Let $G$ be a graph with $k:=\domg(G,X)$. If $\ell\geq k+1$, then
		\[
		\score(G,X,\ell)<\frac{k}{k+1}\ell+1
		\]
	\end{proposition}
	
	For the proof of Proposition~\ref{proposition upper bound on score constant factor more slowly than pallete size}, we need the following definition. For $s\in\mathbb{N}$, let $\mathrm{mod}^*_s:\mathbb{Z}\to [s]$ be the function defined by
	\[
	\mathrm{mod}^*_s(x)=\begin{cases}
		s&\text{ if }s\mid x,\\
		(x\mod s)&\text{ if }s\nmid x.
	\end{cases}
	\]
	I.e. $\mathrm{mod}^*_s$ returns the remainder of $k$ when divided by $s$, except when $k$ is a multiple of $s$, then it returns $s$.
	
	\begin{proof}
		Let $s$ be the largest integer such that $s(k+1)\leq \ell$. We claim that $\score(G,X,\ell)\leq \ell-s$. Indeed, consider the following Bob strategy: We will play two games simultaneously, the main game, a $(G,X,\ell)$-score game between Bob and Alice, and an auxiliary game, a $(G,X,k+1)$-game between players Bob2 and Alice2. Whenever Alice plays color $c$ at vertex $v$, Alice2 will play $\mathrm{mod}^*_{k+1}(c)$ at vertex $v$. Bob2 will play a winning strategy in the $(G,X,k+1)$-game, and Bob will copy Bob2.
		
		With Bob and Bob2 following the above strategies, at the end of the auxiliary game, one of the $k+1$ colors does not end up dominating. In the main game, this corresponds to at least $s$ colors not dominating, completing the proof of our claim.
		
		Now, note that by the definition of $s$, we have that $\ell<(s+1)(k+1)$ or equivalently $s>\frac{\ell}{k+1}-1$. Thus,
		\[
		\score(G,X,\ell)\leq \ell-s<\ell-\left(\frac{\ell}{k+1}-1\right)=\frac{k}{k+1}\ell+1.
		\]
	\end{proof}
	
	\subsection{Order of Play}
	
	In this section we will explore how different two games can be by switching who goes first.
	
	\begin{theorem}\label{theorem gaps between Alice first and Bob first}
		For any graph $G$, we have that
		\[
		\frac{\domg(G,A)-1}{2}\leq \domg(G,B)\leq 2\cdot \domg(G,A)+1,
		\]
		and
		\[
		\frac{\domg(G,B)-1}{2}\leq \domg(G,A)\leq 2\cdot \domg(G,B)+1,
		\]
	\end{theorem}
	
	\begin{proof}
		First, we show $\domg(G,B)\leq 2\cdot \domg(G,A)+1$. Let $k:=\domg(G,A)$. We claim Bob wins the $(G,B,2k+2)$-game. Consider the following Bob strategy: We will play two games simultaneously, the main game, a $(G,B,2k+2)$-game with Alice and Bob, along with the auxiliary game, an $(G,A,2k+2)$-score game with players Alice2 and Bob2, where Bob2 is playing optimally.
		
		In the main game, on Bob's first turn, they choose an arbitrary vertex $v$ to color $1$, we will call $v$ the \textbf{special vertex}. Now, on Alice's turn, Alice and Alice2 make the same moves in the main game and the auxiliary game (i.e. Alice2 copies Alice). Then, on Bob's turn, first Bob2 plays (optimally) in the auxiliary game, and Bob does the following.
		\begin{itemize}
			\item If Bob2 played at the special vertex $v$, then Bob chooses some new uncolored vertex $u$ (if such a vertex exists), and colors it $1$, and labels it as the new special vertex ($v$ is no longer labeled as special).
			\item Otherwise, Bob copies Bob2.
		\end{itemize}
		In each round, there is always exactly $1$ special vertex, and the only differences in the two game boards is that all the vertices that are special or ever been special are colored $1$. This continues until the game ends. In the auxiliary game, by Proposition~\ref{proposition upper bound on score constant factor more slowly than pallete size}, we have that the final score is strictly less than $\frac{k}{k+1}(2k+2)+1=2k+1$, i.e. the score is at most $2k$, so there are at least two colors in the auxiliary game which are not dominating, and one of those colors is not the color $1$, so in the main game, this color is still not dominating, and thus Bob won. This gives that $\domg(G,B)\leq 2\cdot \domg(G,A)+1$.
		
		We will proceed similarly to show that $\domg(G,A)\leq 2\cdot \domg(G,B)+1$. Let $k':=\domg(G,B)$, and consider the following Bob strategy in the $(G,A,2k+2)$-game: We again play two games, the main game, a $(G,A,2k+2)$-game between Alice and Bob, along with an auxiliary game, a $(G,B,2k+2)$-score game between Alice2 and Bob2, where Bob2 will be playing optimally. To initialize, Alice makes a move in the main game, say Alice colors the vertex $v$ the color $1$, we label $v$ as the special vertex and then we begin playing both games simultaneously from here on out. On Bob's turn, first Bob2 makes an optimal move (Bob2 did not see Alice's initial move). Then Bob makes a move according to the following.
		\begin{itemize}
			\item If Bob2 played at the currently special vertex $v$, then Bob arbitrarily chooses some uncolored $u$ (if such a vertex exists), and plays the color $1$ at $u$. $u$ becomes the new special vertex.
			\item Otherwise, Bob copies Bob2.
		\end{itemize}
		Then on Alice's turn, Alice and Alice2 play identically (Alice2 copies Alice).
		
		At any stage in the game, the only differences between the game boards are that the special vertices in the main game are all colored $1$, whereas in the auxiliary game they may be colored differently. Once the game ends, by Proposition~\ref{proposition upper bound on score constant factor more slowly than pallete size}, similar to the first case, the final score in the auxiliary game is at most $2k$, so in particular there are at least two colors which are not dominating, one of which is not the color $1$. This gives us that in the main game there is a color which is not dominating, and so Bob wins the main game, and thus $\domg(G,A)\leq 2\cdot \domg(G,B)+1$.
	\end{proof}
	
	It is worth noting that the proof of Theorem~\ref{theorem gaps between Alice first and Bob first} can be improved if we have better bounds on the score game. In particular, all we needed in the proof is a value $\ell$ such that $\score(G,X,\ell)\leq \ell-2$. If it ends up that $\score(G,X,k)\leq \domg(G,X)$ (see Question~\ref{question score increase}), then the proof above could be strengthened to imply that $|\domg(G,A)-\domg(G,B)|\leq 2$. Even if we could prove this, it is unclear if this is tight; all known examples have the domatic game numbers differing by at most $1$ when we change who goes first. This leads us to the following conjecture.
	
	\begin{conjecture}\label{conjecture changing who goes first}
		For all graphs $G$,
		\[
		|\domg(G,A)-\domg(G,B)|\leq 1.
		\]
	\end{conjecture}
	
	It is worth noting that the difference between $\domg(G,A)$ and $\domg(G,B)$ can be $1$ in either direction. For example, it was shown in~\cite{HR2025} that $\domg(\ell K_2,A)=1$ and $\domg(\ell K_2,B)=2$. In the other direction, let $H$ denote the $5$ vertex graph formed by adding a pendant vertex to a vertex in $K_4$. Then it is straightforward to calculate that $\domg(H,A)=2$ while $\domg(H,B)=1$.
	
	\section{Subgraphs and Graph Operations}\label{section subgraphs}
	
	Hartnell and Rall asked about how the domatic game number is affected by vertex and edge deletion. We partially answer the edge deletion question below.
	
	\begin{observation}
		Let $G$ be a graph and let $e\in E(G)$. Then
		\[
		\domg(G,X)\geq \domg(G-e,X),
		\]
		and for any $\ell\in\mathbb{N}$,
		\[
		\score(G,X,\ell)\geq \score(G-e,X,\ell).
		\]
	\end{observation}
	
	\begin{proof}
		Let $k:=\domg(G,X)$. Bob has a winning strategy for the $(G,X,k+1)$-game on $G$, the same strategy will win the $(G-e,X,k+1)$-game. Similarly, Bob has a strategy to keep the score of the $(G,X,\ell)$-score game to at most $\score(G,X,\ell)$, the same strategy will bound the score of the $(G-e,X,\ell)$-game.
	\end{proof}
	
	A natural question to ask is how much can the domatic game number drop when deleting an edge. It seems feasible that the number to drop by at most $2$ due to the local nature of edge deletion, however we do not state this as a conjecture as it is unclear if there could be global strategy effects from the edge deletion. In the score version, we can make such a claim.
	
	\begin{proposition}
		Let  $G$ be a graph and let $e\in E(G)$. Then for any $\ell\in\mathbb{N}$, 
		\[
		\score(G-e,X,\ell)\geq \score(G,X,\ell)-2.
		\]
	\end{proposition}
	
	\begin{proof}
		Let $e=v_1v_2$, and let us give an Alice strategy. We will play two games simultaneously, The main game will be Alice and Bob playing the $(G-e,X,\ell)$-score game and players Alice2 and Bob2 will play an auxiliary $(G,X,\ell)$-score game. Alice2 will play according to an optimal strategy. Bob2 will copy the moves Bob plays in the $(G-e,X,\ell)$-score game, and our Alice strategy will have copy the moves Alice2 plays in the $(G,X,\ell)$-score game. Since Alice2 used an optimal strategy, the $(G,X,\ell)$-score game resulted in a score of at least $\score(G,X,\ell)$. At the end of $(G-e,X,\ell)$-score game, all the colors that were dominating in the $(G,X,\ell)$ score game are still dominating, with the exception of possibly the colors on the endpoints of $e$, giving Alice total score at least $\score(G,X,\ell) -2$.
	\end{proof}

	We are able to say less about vertex deletion. It is straightforward to see that it is not monotone, for example if $G=K_{n-1}\sqcup K_1$ is the disjoint union of a large clique and an isolated vertex, the domatic game number can go up by close to $n/2$ by deleting the isolated vertex, and on the other hand, deleting vertices eventually drops the domatic game number to $1$. Similar to the edge deletion case, it seems reasonable to guess that the domatic game number will decrease by at most $1$ from deleting a vertex, due to the local effect, however it is unclear if this is the case.
	
	In the score version, we have the following.
	
	\begin{proposition}
		Let $G$ be a graph, $v\in V(G)$, and $\ell\in\mathbb{N}$. Then,
		\[
		\score(G-v,A,\ell)\geq \score(G,B,\ell) -1.
		\]
	\end{proposition}
	
	\begin{proof}
		We will give a strategy for Bob in the $(G,B,\ell)$-score game (the main game) played between players Alice and Bob that keeps the score to at most $\score(G-v,A,\ell)+1$, which will suffice to prove the result. We will also run an auxiliary $(G-v,A,\ell)$-score game with players Alice2 and Bob2. In the main game, Bob will start by coloring vertex $v$ color $1$. After this, Alice2 will copy all of Alice's moves, Bob2 will respond optimally in the auxiliary game, and Bob will copy Bob2's moves.
		
		Since Bob2 is playing optimally, the score of the auxiliary game is at most $\score(G-v,A,\ell)$. The only color that may be dominating in the main game that was not in the auxiliary game is the color $1$ (as $v$ is colored $1$), so at most the score in the main game is $\score(G-v,A,\ell)+1$.
	\end{proof}
	
	\subsection{Unions of Graphs}
	
	In this section we explore how the domatic game number may change under certain graph unions. In particular, we seek to understand when you can ``glue'' two graphs together, while still understanding the domatic game number. More precisely, given a graph $H$, a subset of vertices $L\subseteq V(H)$, and another graph $K$ such that $V(H\cap K)=L$, we'd like to understand the domatic game number of $H\cup K$ in terms of $H$, $K$ and $L$. Of particular interest is when we can bound the domatic game number based only on information about $H$ and $L$, and not $K$ (aside from the fact that $V(H\cap K)=L$).
	
	We start with some technical definitions, which will be useful for describing situations similar to the above, where we can upper bound the domatic game number of $H\cup K$ (i.e. provide a Bob strategy).
	
	\begin{definition}
		Let $H$ be a graph, $L\subseteq V(H)$, and $\ell\in\mathbb{N}$. We say the pair $(H,L)$ is \textbf{$(X,\ell)$-BobGood} if there exists a winning Bob strategy in the $(H,X,\ell)$-game such that at the end of any game played with this strategy, there exists a vertex $x\in V(H)\setminus L$ that Bob wins at (i.e. there exists a color $c\in [\ell]$ that does not appear anywhere on $N_H[x]$). Such a strategy will be called a \textbf{$(H,L,X,\ell)$-BobGood} strategy.
	\end{definition}
	
	Roughly, the pair $(H,L)$ is $(X,\ell)$-BobGood if Bob can win the $(H,X,\ell)$-game even if Alice is given the extra advantage of not having to make sure the vertices in $V(L)$ are actually dominated.
	
	\begin{lemma}\label{lemma Bob good}
		Let $H$ and $K$ be graphs with $L:=V(H\cap K)$, and let $\ell\in\mathbb{N}$. Assume $(H,L)$ is $(X,\ell)$-BobGood. Then the following hold.
		\begin{enumerate}[label=(B\arabic*)]
			\item If $|V(K)|-|L|$ is even, then $\domg(H\cup K,X)\leq \ell-1$.\label{statement bob good even}
			\item If $|V(K)|-|L|$ is odd and $X=A$, then $\domg(H\cup K,B)\leq \ell-1$.\label{statement bob good odd}
		\end{enumerate}
	\end{lemma}
	
	\begin{proof}
		Let us start with~\ref{statement bob good even}. We will give a Bob strategy in the $(H\cup K,X,\ell)$-game. If on the previous turn, Alice played on $V(K)\setminus L$, then Bob will arbitrarily choose a vertex in $V(K)\setminus L$ and play the color $1$. Otherwise, Bob plays on $H$ with an $(H,L,X,\ell)$-BobGood strategy. This $|V(K)|-|L|$ is even, this is a valid strategy (i.e. Bob is always able to play in $V(K)\setminus L$ after Alice does). Since Bob played according to a $(H,L,X,\ell)$-BobGood strategy on $H$ a vertex in $V(H)\setminus L$ ends up not being dominated by one of the $\ell$ colors, so Bob wins. Thus, $\domg(H\cup K,X)\leq \ell-1$.
		
		Now let us prove~\ref{statement bob good odd}. We will give a Bob strategy in the $(H\cup K,B,\ell)$-game. Bob will play his first move in $V(K)\setminus L$. If on the previous turn Alice played in $V(K)\setminus L$, Bob will play in $V(K)\setminus L$, and otherwise Bob will play in $V(H)$, unless every vertex in $V(H)$ already has a color, in which case Bob will play in $V(K)\setminus L$. Whenever Bob plays in $V(K)\setminus L$, he arbitrarily chooses a vertex and colors it $1$, whereas when he plays in $V(H)$, he plays according to a $(H,L,X,\ell)$-BobGood strategy. Since Bob plays a BobGood strategy on $H$, this results in a vertex in $V(H)\setminus L$ not being dominated by some color, so Bob wins. This gives us that $\domg(H\cup K,B)\leq \ell-1$.
	\end{proof}
	
	One example of using BobGood sets to bound the game domatic number is in graphs with adjacent vertices of degree $2$ - If $x$ and $y$ are the degree one vertices in $P_4$, it is is straightforward to show that $(P_4,\{x,y\})$ is $(X,2)$-BobGood, and similarly if $x$ and $y$ are adjacent vertices in $C_4$, then $(C_4,\{x,y\})$ is also $(X,2)$-BobGood. Using this, along with Lemma~\ref{lemma Bob good} we immediately get the following.
	
	\begin{proposition}
		Let $G$ be a graph with an even number of vertices which contains a pair of adjacent degree $2$ vertices. Then $\domg(G,X)=1$.
	\end{proposition} 
	
	BobGood sets in general may not be easy to identify, however dominating vertices are natural candidates for BobGood sets.
	
	\begin{lemma}\label{lemma dominating bobgood}
		Let $H$ be a graph with $dom(H,X)=\ell$, and let $L\subset V(G)$ be a collection of dominating vertices of $G$. Then, $(H,L)$ is $(X,\ell+1)$-BobGood.
	\end{lemma}
	
	\begin{proof}
		Consider an optimal Bob strategy in the $(H,X,\ell+1)$-game, call it STRAT. We claim that at the end of any $(H,X,\ell+1)$-game where Bob plays according to STRAT, there is a vertex $v\in V(H)\setminus L$ and a color $c\in [\ell+1]$ such that $v$ is not dominated by the color $c$. Indeed, since Bob wins the $(H,X,\ell+1)$-game, Bob wins at some vertex $x\in V(H)$. If $x\in V(H)\setminus L$, we have proven our claim. Otherwise, $x\in L$, and since $x$ is a dominating vertex, there must be some color $c\in [\ell+1]$ that does not appear at all on $H$. so for any $v\in V(H)\setminus L$, $v$ is not dominated by $c$. Thus, in any case, we find that Bob wins at a vertex in $V(H)\setminus L$ when playing according to STRAT. Thus, $(H,L)$ is $(X,\ell+1)$-BobGood.
	\end{proof}
	
	We now focus on lower bounds for graph unions. Again, we need a few technical definitions.
	
	\begin{definition}
		Let $H$ be a graph, $L\subseteq V(H)$, and $\ell\in\mathbb{N}$. We say the pair $(H,L)$ is \textbf{$(X,\ell)$-AliceGood} if there exists a winning Alice strategy in the $(H-L,X,\ell)$-game. Such a strategy will be called a \textbf{$(H,L,X,\ell)$-AliceGood} strategy.
	\end{definition}
	
	In essence, $(H,L)$ is $(X,\ell)$-AliceGood if Alice can win on $V(H)\setminus L$.

	\begin{lemma}\label{Lemma alicegood}
		Let $H$ and $K$ be graphs with $L:=V(H\cap K)$, and let $\ell\in\mathbb{N}$. Assume $(H,L)$ is $(X,\ell)$-AliceGood and that $dom(K,Y)\geq \ell$. Then the following hold.
		\begin{enumerate}[label=(A\arabic*)]
			\item If $X=Y=B$ and $|V(K)|$ and $|V(H)\setminus L|$ are both even, then $\domg(H\cup K,B)\geq \ell$.\label{statement alicegood x=y=b}
			\item If $X=A$, $Y=B$, $|V(K)|$ is even and $|V(H)\setminus L|$ is odd, then $\domg(H\cup K,A)\geq \ell$.\label{statement alicegood x=a y=b}
			\item If $X=B$, $Y=A$, $|V(K)|$ is odd and $|V(H)\setminus L|$ is even, then $\domg(H\cup K,A)\geq \ell$.\label{statement alicegood x=b y=a}
		\end{enumerate}
	\end{lemma}
	
	\begin{proof}
		Consider the following Alice Strategy for the $(H\cup K,Z,\ell)$ game (where $Z=B$ if we are in situation~\ref{statement alicegood x=y=b}, and $Z=A$ in situations~\ref{statement alicegood x=a y=b} and~\ref{statement alicegood x=b y=a}). If Bob played on $V(K)$ in the previous round, then Alice will also play on $V(K)$ according to a winning $(K,Y,\ell)$-game strategy. If Bob played on $V(H)\setminus L$ in the previous round, then Alice will also play on $V(H)\setminus L$ according to a $(H,L,X,\ell)$-AliceGood strategy. If we are in situation~\ref{statement alicegood x=a y=b}, on the first move Alice plays in $V(H)\setminus L$ (following the $(H,L,X,\ell)$-AliceGood strategy), where if we are in situation~\ref{statement alicegood x=b y=a}, on the first move Alice plays in $V(K)$ (with a winning $(K,Y,\ell)$-game strategy).
		
		By the conditions on the parity of $|V(K)|$ and $|V(H)\setminus L|$, the above is a valid strategy in any of the three situations (i.e. there is never a time before the game ends when Alice wants to make a move in a set, but there are no uncolored vertices in that set), and in all three cases the game ends with all vertices in $V(K)$ and in $V(H)\setminus L$ being dominated by all colors, so Alice wins in all cases.
	\end{proof}
	
	As an example of how to apply AliceGood subgraphs, we present the following proposition, showing that certain unions of cliques have bounded domatic game number. Given an integer $t\geq 1$ will say a graph $G$ is a \textbf{$K_{2t+1}$-tree} if there exists a collection $K^{(1)}$, $K^{(2)}$,\dots $K^{(s)}$ of copies of $K_{2t+1}$ such that $G=\bigcup_{i=1}^s K^{(i)}$ and for any $i\in [s]\setminus\{1\}$, $|V(\bigcup_{j=1}^{i-1}K^{(j)})\cap V(K^{(i)})|=1$.
	
	\begin{proposition}
		Let $G$ be a $K_{2t+1}$-tree. Then
		\[
		\domg(G,A)=t+1.
		\]
	\end{proposition}
	
	\begin{proof}
		To prove that $\domg(G,A)\geq t+1$, we will use induction on the number of copies $s$ of $K_{2t+1}$ there are in $G$. If $s=1$, then $G=K_{2t+1}$, and in~\cite{HR2025}, it was shown that $\mathrm{dom}(K_{2t+1},A)=t+1$. Now assume $s\geq 2$. Let $G$ be a $K_{2t+1}$-tree with cliques $K^{(1)}$, $K^{(2)}$, \dots $K^{(s)}$ and let $G':=\bigcup_{i=1}^{s-1}K^{(i)}$, noting that $G'$ is also a $K_{2t+1}$-tree, and by induction we have that $\domg(G',A)=t+1$. For any vertex $v\in V(K_{2t+1})$, we have that $\domg(K_{2t+1}-v,B)=\domg(K_{2t},B)=t+1$, where the second equality was shown in~\cite{HR2025}, so $(K_{2t+1},\{v\})$ is $(A,t+1)$-AliceGood. Since $|V(G')|$ is odd and $|V(K_{2t+1})\setminus \{v\}|$ is even, by Lemma~\ref{Lemma alicegood}, $\domg(G_{k+1},A)\geq n+1$. 
		
		To see that $\domg(G,A)\leq t+1$, notice that taking $v\in V(K^{(s)})\cap V(G')$, we have that $(K^{(s)},v)$ is $(A,t+2)$-BobGood by Lemma~\ref{lemma dominating bobgood}. Since $|V(G')\setminus \{v\}|$ is even, $\domg(G,A)\leq t+2-1$ by Lemma~\ref{lemma Bob good}. Therefore, $\domg(G,A)=t+1$.
	\end{proof}
	
	\section{Concluding Remarks}\label{section concluding remarks}
	
	Theorem~\ref{theorem main theorem on degree} is perhaps our most significant theorem, giving a lower bound for $\domg(G,X)$ on the order of $\frac{\delta(G)}{\log n}$ for a graph $G$ of order $n$. Comparing this to Theorem~\ref{theorem domatic number minimum degree maximum degree lower bound}, it is natural to wonder if this can be improved to $\frac{\delta(G)}{\log \Delta(G)}$. The $\log n$ in our bound stems directly from the use of Theorem~\ref{theorem Erdos selfridge}, so any improvement over Theorem~\ref{theorem main theorem on degree} would need to avoid this somehow. We do not know of examples with domatic game number smaller (in order of magnitude) than $\frac{\delta(G)}{\log \Delta(G)}$, so it is unclear if Theorem~\ref{theorem main theorem on degree} is best possible. We leave this as an open question to the reader.
	
	\begin{question}
		Does there exist a sequence of graphs $G_n$ such that
		\[
		\domg(G_n,X)=o\left(\frac{\delta(G_n)}{\log \Delta(G_n)}\right)?
		\]
		If so, is Theorem~\ref{theorem main theorem on degree} tight?
	\end{question}
	
	Another natural avenue to explore is the behavior of the domatic score game on $G$ when the palette size is larger than $\domg(G,X)$. Proposition~\ref{proposition score does not decrease much} shows us that the score cannot be much smaller than the domatic game number, while Proposition~\ref{proposition upper bound on score constant factor more slowly than pallete size} shows that if the score increases as the palette size increases, it does so a constant factor slower. However, we do not know if the score can even exceed the domatic game number.
	
	\begin{question}\label{question score increase}
		For any graph $G$ and $k\geq \domg(G,X)$, is it the case that is it the case that
		\[
		\score(G,X,k)\leq \domg(G,X)?
		\]
		Is there a monotonic relationship between $\score(G,X,k)$ and $\score(G,X,k+1)$ when $k>\domg(G,X)$?
	\end{question}
	
	Part of the motivation for answering Question~\ref{question score increase} is that improvements on Proposition~\ref{proposition upper bound on score constant factor more slowly than pallete size} can further imply better bounds than those given in Theorem~\ref{theorem gaps between Alice first and Bob first} on the effect of changing who goes first. In particular, we would be very interested in a proof (or disproof) of Conjecture~\ref{conjecture changing who goes first}.
	
	The final direction for further exploration we will mention here involves vertex and edge deletion. Originally asked by Hartnell and Rall~\cite{HR2025}, we provide a partial answer in some of our results in Section~\ref{section subgraphs}, however there is still a lot left unknown. In particular this question seems difficult due to the effect of changing the palette size - making common ``strategy borrowing'' arguments not work. A few specific things of interest are summarized in the following question.
	
	\begin{question}
		How much can the domatic game number drop from deleting a single edge? A single vertex? Is there a reasonable characterization of the situations where the domatic game number can increase when deleting a vertex?
	\end{question}
	
	\bibliographystyle{amsplain}
	\bibliography{bib}{}

\end{document}